\documentclass[12pt]{amsart}
\usepackage{amsmath,
amsfonts,
amsthm,
amssymb,
bbold,
mathrsfs}

\newcommand{\dom}{\operatorname{dom}}
\newcommand{\comment}[1]{}

\newcommand{\lex}{<_{\textrm{lex}}}
\newcommand{\PFA}{\mathrm{PFA}}
\newcommand{\CH}{\mathrm{CH}}
\newcommand{\GCH}{\mathrm{GCH}}

\newcommand{\ZFC}{\mathrm{ZFC}}

\theoremstyle{plain}
\newtheorem{thm}{Theorem}[section]
\newtheorem{lem}[thm]{Lemma}
\newtheorem{prop}[thm]{Proposition}
\newtheorem{cor}[thm]{Corollary}
\newtheorem{fact}[thm]{Fact}

\newtheorem{question}[thm]{Question}

\theoremstyle{definition}
\newtheorem{defn}[thm]{Definition}

\newtheorem{notation}[thm]{Notation}

\begin{document}

\author[H. Lamei Ramandi]{Hossein Lamei Ramandi}

\address{Department of Mathematics \\ University of Toronto,
Toronto \\ Canada}

\title[A New Non-$\sigma$-Scattered Order]{A New Minimal Non-$\sigma$-Scattered Linear Order}


\email{{\tt hossein@math.toronto.edu}}

\begin{abstract}
We will show it is consistent with $\GCH$ that there is a  minimal non-$\sigma$-scattered linear order 
which does not contain any real or Aronszajn type. In particular the 
assumption $\PFA^+$ in the main result of \cite{no_real_Aronszajn} is necessary, 
and there are other obstructions than real and Aronszajn types to the sharpness
of Laver's theorem in \cite{Fraisse_otp_conj}. 
\end{abstract}

\maketitle

\section{Introduction}

Fra\"\i ss\'e in \cite{Fraisse}, conjectured that every descending sequence of countable order types is 
finite, and every antichain of countable order types is finite. That is, the class of countable linear orders is 
\emph{well quasi ordered}. Laver confirmed the conjecture by proving a stronger statement.
\begin{thm}\cite{Fraisse_otp_conj} \label{laver}
The class of $\sigma$-scattered linear orders  is well quasi ordered. In particular every descending chain
of $\sigma$-scattered linear orders is finite.
\end{thm}
Here the class of linear orders is considered with the quasi order of embeddability. Recall that a linear
order $L$ is said to be scattered if it does not contain a copy of the rationals. 
$L$ is called $\sigma$-scattered 
if it is a countable union of scattered linear orders.
At the end of his paper, Laver asks about the behavior of non-$\sigma$-scattered 
linear orders under embeddability.
For instance 
\emph{to what extent Laver's theorem is sharp?} If the answer to this question is independent of $\ZFC$, 
\emph{what are the 
obstructions to the sharpness of Laver's theorem?}

Not very long
after Laver proved Theorem \ref{laver}, various theorems in the direction of showing that Laver's theorem is consistently not sharp were proved. Baumgartner in \cite{reals_isomorphic},
showed that it is consistent that all $\aleph_1$-dense sets 
of the reals are isomorphic.
In the same paper, he mentions that one can add all $\aleph_1$ sized subsets of the reals to 
the class of all $\sigma$-scattered linear orders in order to obtain a class $\mathcal{L}$
of linear orders such that $\mathcal{L}$ is
strictly larger than the class of $\sigma$-scattered linear orders, $\mathcal{L}$ is closed under taking suborders 
and it is consistent that $\mathcal{L}$ is well quasi ordered.

Another result in the direction of 
``Laver's theorem is consistently not sharp" is due to Abraham and Shelah.
In \cite{club_isomorphic}, they showed that $\PFA$, the 
proper forcing axiom, implies that every two non-stationary Countryman lines are either isomorphic or reverse 
isomorphic. An Aronszajn line $A$ is said to be non-stationary if there is a continuous increasing sequence
$\langle A_\xi : \xi \in \omega_1 \rangle$ of countable subsets of $A$ which covers $A$
such that for each $\xi \in \omega_1$
no maximal interval of $A \setminus A_\xi$ has a least or greatest element.
Since every Countryman line contains an uncountable non-stationary suborder,
one can even considers a larger class of linear orders than what Baumgartner and 
Laver considered and still have a class of linear orders which is consistently well quasi ordered and 
which is closed under taking suborders. 
Later Martinez-Ranero 
in \cite{A-line_wqo} showed that under $\PFA$ the class of all Aronszajn lines is 
well quasi ordered.

Baumgartner seems to be the first person who considered the other side of Laver's question, i.e, to what 
extent is Laver's theorem sharp?
In \cite{new_class_otp}, he introduces a class of non-$\sigma$-scattered linear orders and
proves in $\ZFC$ that his examples are not minimal with respect to being non-$\sigma$-scattered. Baumgartner's 
example can be described as follows. Let $L=\{C_\alpha : \alpha \in S \}$ ordered lexicographically, where 
$S$ is a stationary subset of $\omega_1$ consisting of limit ordinals and $C_\alpha$ is a cofinal sequence 
in $\alpha$ that has order type $\omega$. 
Baumgartner's example $L$ has the property that a suborder $\{C_\xi : \xi \in A \}$ is $\sigma$-scattered if and 
only if
$A$ is not stationary. This together with pressing down lemma implies that
if $f:L \longrightarrow L$ is an embedding
then the set $\{\xi \in S: f(C_\xi) \neq C_\xi \}$ is not stationary. Therefore if 
$S_0 \subset S$ is such that $S\setminus S_0$ and $S_0$ are stationary then $L$ does not embed into the 
linear order corresponding to $S_0$. In this paper \emph{Baumgartber type} refers to Baumgartner's examples or the 
revers of them.

The behavior of Baumgartner types inspired Ishiu and Moore to 
generalize the situation above for a broader class of linear 
orders and prove the following theorem.
\begin{thm} \cite{no_real_Aronszajn} \label{no_real_Aronszajn}
$\PFA^+$ implies that every minimal non-$\sigma$-scattered linear order is either a real or a Countryman type.
\end{thm}
In other words under $\PFA^+$, the only obstructions to the sharpness of Laver's theorem are real and Countryman
types. This breakthrough should be considered with the 
following result. 
\begin{thm} \cite{minimal_unctbl_types} \label{omega_1}
It is consistent with $\CH$ that $\omega_1$ and $\omega_1^*$ are the only linear orders that are minimal 
with respect to being uncountable. 
\end{thm}

Later the methods in \cite{no_real_Aronszajn} and \cite{minimal_unctbl_types} were used to prove Laver's 
theorem is sharp, i.e, it is impossible to improve the theorem in $\ZFC$. 

\begin{thm} \cite{first} \label{sharp}
If there is a supercompact cardinal,
then there is a forcing extension which satisfies $\CH$ in which there are no minimal
non-$\sigma$-scattered linear orders.
\end{thm}
Note that all of the results proving that Laver's theorem is consistently 
not sharp were based on the consistency of the 
minimality of real types or Aronszajn types. So it is natural to ask the following question.
\begin{question}
Does any minimal non-$\sigma$-scattered linear order have to be real or Aronszajn type?
\end{question} 
This question is also important from the point of view of the work involved in proving Theorems \ref{no_real_Aronszajn}, \ref{omega_1} and \ref{sharp}. An affirmative answer to this question would assert that
the assumption $\PFA^+$ would  not be needed in order to obtain the results in  \cite{no_real_Aronszajn}.  
Consequently, the model Moore came up with in order to prove
Theorem \ref{omega_1} would already satisfy ``Laver's thorem is sharp." Therefore the work in \cite{first} as well 
as the large cardinal assumption would not be  needed to prove Theorem \ref{sharp}.
In this paper we will provide a negative answer to this question. In particular 
real and Aronszajn types are not the only 
possible obstructions to the sharpness of Laver's theorem.  
\begin{thm} \label{main}
It is consistent with $\GCH$ that there is a non-$\sigma$-scattered linear order $L$ which contains no real or 
Aronszajn 
type and is minimal with respect to not being $\sigma$-scattered.
\end{thm}
Moreover, Theorem \ref{main} is related the following question which is due to Galvin.
\begin{question} \cite[problem 4]{new_class_otp}
Is there a linear order which is minimal with respect to not being $\sigma$-scattered and which has the property
that all of its uncountable suborders contain a copy of $\omega_1$?
\end{question}
Note that a consistent negative answer 
is already given by Theorem \ref{no_real_Aronszajn}. Theorem \ref{main} does not answer Galvin's question because
the linear order we introduce, has a lot of copies of $\omega_1^*$.

This paper is organized as follows. Section \ref{back} reviews some notations, definitions and facts regarding 
linear orders. Section \ref{tree} is devoted to constructing a specific Kurepa tree that is a suitable candidate 
for having suborders that witness Theorem \ref{main}. We also show that this tree contains a lot of 
non-$\sigma$-scattered linear orders which become $\sigma$-scattered in order to witness the main result. 
In Section \ref{embeddings} we introduce the posets that add isomorphisms we need.
Section \ref{last} finishes the proof of Theorem \ref{main}.

\section{Preliminaries} \label{back}

This section is devoted to some background, notation and definitions regarding trees, linearly ordered sets,
forcings and their iterations. More discussion can be found in \cite{no_real_Aronszajn}, 
\cite{second}, \cite{first},  and \cite{proper_forcing}.

To avoid ambiguity we fix some terminology and notations.
An \emph{$\omega_1$-tree} is a tree which is of height $\omega_1$, has countable levels and does not branch at limit 
heights, i.e. if $s,t$ are of the same limit height and have the same predecessors then they are equal.
A \emph{branch} of a tree $T$ is a chain in $T$ which intersects all levels.
An $\omega_1$-tree $T$ is called \emph{Aronszajn} if it has no branches. It is called \emph{Kurepa} if it has 
at least $\omega_2$ many branches.

For a tree $T$ and $t \in T$, $T(t)$ is the collection of all elements of $T$ that are comparable with $t$.
If $T$ is a tree and $A$ is  a set of ordinals, by $T\upharpoonright A$ we mean $\{t \in T: \textrm{ht}(t)\in A \}$, 
with the order inherited from $T$.
If $S,T$ are trees of height $\kappa$ and $C\subset \kappa$ is a club and 
$f:T\upharpoonright C \longrightarrow S\upharpoonright C$ is one to one, level and order preserving
then $f$ is called a \emph{club embedding} from $T$ to $S$.  
$\mathcal{B}(T)$ refers to the collection of all branches of $T$. 
If $L$ is a linearly ordered set, $\hat{L}$ denotes the completion of $L$. Formally $\hat{L}$ consists of 
all Dedekind cuts of $L$.

The following few definitions and facts give a characterization of $\sigma$-scatteredness which we use 
in the proof of Theorem \ref{main}. They also generalize the behavior of Baumgartner types that causes them
to be non-minimal. 
We will use this to show that the generic tree that we build in section \ref{tree}
has suborders that are obstructions to minimality.
\begin{defn} \cite{no_real_Aronszajn}
Assume $L$ is a linear order and $Z$ is a countable set. We say $Z$ \emph{captures} $x \in L$ if there is a 
$z \in Z\cap \hat{L}$ such that there is no element of $Z\cap L$ strictly in between $x$ and $z$.
\end{defn}
\begin{fact} \cite{no_real_Aronszajn}
Suppose $L$ is a linear order and $\kappa$ is a regular large enough cardinal. If $M$ is a countable 
elementary submodel of $H_\kappa$ such that $L\in M$ and $x\in L\setminus M$, then 
$M$ captures $x \in L$ iff it there is a unique $z \in \hat{L}\cap M$ such that there is no element of 
$M \cap L$ strictly in between $x$ and $z$. In this case we say $M$ captures $x$ via $z$.
\end{fact}
\begin{defn} \cite{no_real_Aronszajn}
Assume $L$ is a linear order. \emph{$\Omega(L)$} 
is the set of all countable $Z\subset \hat{L}$ which capture all 
elements of $L$. $\Gamma(L)=[\hat{L}]^\omega \setminus \Omega(L)$.
\end{defn}
\begin{prop} \cite{no_real_Aronszajn} \label{char}
A linear order $L$ is $\sigma$-scattered iff $\Gamma(L)$ is not stationary in $[\hat{L}]^\omega$.
\end{prop}
\begin{defn} \cite{no_real_Aronszajn}
Assume $L$ is a linear order, $x \in L,$ and $M$ is a countable elementary submodel of $H_\theta$ where 
$\theta> 2^{|L|}$ is a large enough regular cardinal. We say $x$ is \emph{internal} to $M$ if there is a club 
$E \subset [L]^\omega$ in $ M$ such that whenever $Z\in M\cap E$, $Z$ captures $x\in L$. 
We say $L$ is \emph{amenable} if for all 
large enough regular cardinals $\theta$, for all countable 
elementary submodels $M$ of $H_\theta$ that contain
$L$ as an element, and for all $x \in L$, $x$ is internal to $M$. 
\end{defn}
The following proposition shows that amenability is what causes 
Baumgartner types and consistently more
linear orders to be non-minimal, see \cite{first}, discussion after the proof of Theorem 3.1.
\begin{prop} \cite{no_real_Aronszajn}
If $L$ is an amenable non-$\sigma$-scattered linear order of size $\aleph_1$, then it is not minimal with respect to 
being non-$\sigma$-scattered. 
\end{prop}
During this paper we consider the invariants $\Omega,$ and $\Gamma$ for trees and linear orders 
with different definitions. The point is that all these definitions coincide modulo an equivalence relation
that is defined here.
\begin{defn} \label{equiv}
Assume $X,Y$ are two uncountable sets and $A,B$ are two collections of countable subsets of $X,Y$ 
such that $\bigcup A = X$ and $\bigcup B =Y$.
We say $A,B$ are equivalent if for all large enough regular cardinal $\theta$ there is a club $E$ of 
countable elementary submodels $M$ of $H_\theta$ such that $M\cap X \in A$ if and only if $M\cap Y \in B$.
\end{defn}
The invariant $\Omega$ together with the equivalence relation mentioned above was used in 
\cite{no_real_Aronszajn}. By the work in \cite{no_real_Aronszajn}, if $L_0\subset L$ and $L$ embedds into 
$L_0$ then $\Omega(L)$ is equivalent to $\Omega(L_0)$. In fact 
the strategy in that work  was to find a suborder $L_0$ of a given non-$\sigma$-scattered 
linear order $L$ such that $\Omega(L_0)$ is stationary and not equivalent to $\Omega(L)$. 
This seems to be the motivation of Problem 5.10 in \cite{no_real_Aronszajn}. The problem asks, assuming that 
$S$ is stationary, is the class of all linear orders $L$ with $\Omega(L) \equiv S$ well quasi ordered?
We will show that even with such a restriction on the $\Omega$ of non-$\sigma$-scattered linear orders 
it is impossible to obtain a well quasi ordered class. Here for linear orders $A$ and $ B$, the linear order consisting
of the disjoint union of $A,B$ in which every element of $A$ is less than every element of $B$ is denoted
by $A \oplus B$.
\begin{prop}
Assume $S\subset \omega_1$ is a stationary set consisting of limit ordinals, and $\{S_i : i \in \omega \}$ is a 
partition of $S$ into infinitely many stationary pieces. Let 
$\langle C_\alpha \subset \alpha: \alpha \in S \rangle$ be a 
collection of cofinal sequences of order type $\omega$.
Let $L=\{C_\alpha : \alpha \in S \}$ and
$L_i=\{C_\alpha : \alpha \in \bigcup_{j\geq i}S_j \}$ ordered with the lex order.
Then the sequence $\langle L \oplus L_i : i \in \omega \rangle$ is a descending chain of linear 
orders and $\Omega(L \oplus L_i) \equiv \Omega(L \oplus L_j)$ for all $i,j$ in $\omega$.
\end{prop}
\begin{proof}
We start with an observation.  Assume $L_X=\{x_\alpha: \alpha \in X\}$ and $L_Y= \{ y_\alpha : \alpha \in Y \}$ are two 
arbitrary Baumgartner types where $X,Y$ are stationary subsets of   $\omega_1$ consisting of limit ordinals and 
$x_\alpha, y_\alpha$
are  increasing cofinal  $\omega$-sequences in $\alpha$.
Suppose $f : L_X \longrightarrow L_Y$ is an embedding. Then 
the set of $\alpha$ with $f(x_\alpha) \neq y_\alpha$ is nonstationary.
In order to see this, note that the set of all $\alpha$ with $f(x_\alpha) = y_\xi$, for some $\xi \in \alpha$, is 
nonstationary by pressing down lemma. Similarly, the set of all $\beta $ with $y_\beta =f( x_\alpha)$, for some 
$\alpha \in \beta$ is nonstationary. Hence $\{x_\alpha : f(x_\alpha) = y_\beta (\exists \beta > \alpha) \}$
is $\sigma$-scattered. Therefore the set of  all $\alpha$ with $f(x_\alpha)= y_\beta$ for some $\beta > \alpha$
is nonstationary as desired.

Now we will show that the sequence $\langle L \oplus L_i : i \in \omega \rangle$ is a descending chain. Assume 
for some $m \in \omega$, $L \oplus L_m$ embeds into $L \oplus L_{m+1}$. Let 
$A=\{a_\alpha : \alpha \in \bigcup_{i \geq m} S_i \}$ and $B=\{b_\alpha : \alpha \in \bigcup_{i\geq m+1} S_i \}$
be disjoint from $L$ and be isomorphic to $L_m$ and $L_{m+1}$ respectively, via the maps $a_\alpha \mapsto C_\alpha$
and $b_\alpha \mapsto C_\alpha$. 
Also let $f:L\oplus  A \longrightarrow L \oplus B$ be an embedding.
Then the sets 
$\{ \alpha : f(C_\alpha) \neq C_\alpha $ and $f(C_\alpha) \neq b_\alpha \}$ and 
$\{\alpha : f(a_\alpha) \neq C_\alpha$ and $f(a_\alpha) \neq b_\alpha \}$ are nonstationary.
Therefore, there is a nonstationary $N \subset \omega_1$ such that for all 
$\alpha \in S_m \setminus N$, 
$f(a_\alpha) =C_\alpha$ and for all $\alpha \in S_m \setminus N$, $f(C_\alpha)=C_\alpha$. This contradicts 
the injectivity of $f$.

If $M$ is a countable elementary submodel of $H_\theta$ $(\theta > 2^{\omega_1})$, $M$ captures
all elements of $L \oplus L_i$ if and only if $M\cap \omega_1 \notin S$. This shows that 
$\Omega(L \oplus L_i) \equiv \Omega(L \oplus L_j)$.
\end{proof}

Assume  $T$ is an $\omega_1$-tree that is equipped with a lexicographic order
such that for all $t \in T$, the set $\{s \in T: \textrm{\textrm{ht}(s)}= \textrm{ht}(t)+1$ and $t<_T s \}$ 
is isomorphic to $\mathbb{Q}$, when it is considered with the lex order of the tree $T$.
Let $L=(T,\lex)$, then $\Omega(L)$ defined here is equivalent to the $\Omega(T)$ defined in 
\cite{second}.
\begin{defn} \cite{second}
Assume $T$ is an $\omega_1$-tree.
$\Omega(T)$ is the set of all countable $Z\subset \mathcal{B}(T)$ with the property that 
for all $t \in T_{\alpha_Z}$ there is a $b \in Z$ with $t \in b$, where
$\alpha_Z = \sup\{b\Delta b':b,b'\in Z\}$. 
\end{defn}

Now we have two notions of capturing:  for linear orders and $\omega_1$-trees. The following
trivial fact asserts that 
in the cases that we are interested in, these two notions coincide. 
\begin{fact}
Assume $T$ is an $\omega_1$-tree equipped with a lex order such that the set of all immediate successors of 
each element of $T$ is isomorphic to $\mathbb{Q}$.
Suppose $\mathcal{B}(T)$ is the collection of all cofinal branches  in $T$, 
$M$ is a countable elementary submodel of $H_\theta$ for some large enough regular $\theta$, and $t\in T$. 
Then  $M$ captures $t$ if and only if 
either $t \in M$ or there is a branch $b \in M\cap \mathcal{B}(T)$ such that $t\Delta b \geq M\cap \omega_1$.
Here $t\Delta b$ is the height of the least element of $b$ that is not a predecessor of $t$.
\end{fact}

The following fact is also easy to check.
\begin{fact}
Assume $L'\subset L$ are linear orders, $x \in L'$ and $M$ is a countable elementary submodel of $H_\theta$
that has $L,L'$ as elements, where $\theta>2^{|\hat{L}|}$ is regular. Then $M$ captures $x$ as an element 
of $L$ iff $M$ captures $x$ as an element of $L'$.
\end{fact}
We will need the following fact in the final section. We include the proof for more clarity.
\begin{fact}
Assume $L$ is a linear order which has size $\aleph_2$, all elements of $L$ have cofinality and coinitiality
$\omega_1$, and $L'\subset L$ is dense and has cardinality $\aleph_1$. Then $L'$ is not $\sigma$-scattered.
\end{fact}
\begin{proof}
Assume $L'$ is $\sigma$-scattered. 
Since all $x \in L$ have cofinality and coinitiality $\omega_1$, there is a scattered suborder $L_0$
of $L'$ whose closure in $L$ has cardinality $\aleph_2$. For $x,y \in L_0$ let $x\sim y$ if there are at most 
$\aleph_1$ many elements of the closure of ${L_0}$ in between $x$ and $y$. Note that there are exactly $\aleph_1$ many 
equivalence classes and between every two distinct equivalence classes there are  
$\aleph_1$ many, equivalence classes. Now let $L_1$ be a suborder of $L_0$ which intersect each equivalence
class at exactly one point. $L_1$ is an infinite dense linear order which contradicts scatteredness
of $L_0$.
\end{proof}

We will be using forcings which are not proper.
The rest of this section is devoted to the facts and lemmas which enable us to show that 
countable support iterations of these forcings are robust enough to preserve cardinals, under mild assumptions
like $\CH$. More discussion can be found in \cite{second} and \cite{proper_forcing}.

For a regular cardinal $\theta$, $H_\theta$ is the collection of all sets of hereditary cardinality less than $\theta$.
We assume $H_\theta$ is equipped with a fixed well ordering without mentioning it.
Assume $\mathcal{P}$ is an arbitrary set  
and $\theta$ is a regular cardinal such that $\mathcal{P}$ and the powerset
of $\mathcal{P}$ are in $H_\theta$. A countable elementary submodel 
$N$ of $H_\theta$ is said to be \emph{suitable} for $\mathcal{P}$
if $\mathcal{P} \in N$. 
If $\mathcal{P}$ is a forcing notion and $\langle p_n : n \in \omega \rangle$ is a decreasing sequence of conditions  
in $\mathcal{P}\cap N$, $\langle p_n : n \in \omega \rangle$
is said to be \emph{$(N,\mathcal{P})$-generic} if 
for all dense subsets $D$ of $\mathcal{P}$ that are in $N$, there is an $n \in \omega$ such that $p_n \in D$.

\begin{defn}
Assume $X$ is uncountable and $S \subset [X]^\omega$ is stationary. A poset $\mathcal{P}$ is 
said to be \emph{$S$-complete}, 
if every descending $(M, \mathcal{P})$-generic
sequence, $\langle p_n: n\in \omega \rangle$ has a lower bound, for all $M$ with $M \cap X \in S$ and
$M$ suitable for $X,\mathcal{P}$.
\end{defn}
\begin{fact} \cite{proper_forcing}
 Assume $X$ is uncountable and $S\subset [X]^\omega$ is stationary. If $\mathcal{P}$ is an $S$-complete
  forcing then it preserves $\omega_1$ and adds no new countable sequences of ordinals.
\end{fact}
\begin{cor} \cite{proper_forcing}
\begin{cor} \label{iteration}
Assume $X$ is uncountable and $S\subset[X]^\omega$ is stationary. Then $S$-completeness is 
preserved under countable support iterations.
\end{cor}
\end{cor}
We will use the following lemmas from \cite{second}.
Note that  no forcing can add a new cofinal branch or Aronszajn subtree to $T$
when $T$ has no Aronszajn subtree and has only countably many cofinal branches. 
\begin{lem} \label{No A subtree} \cite{second}
Assume $T$ is an $\omega_1$-tree 
which has uncountably many cofinal branches and which has no Aronszajn subtree 
in the ground model $\mathbf{V}$. Also assume
 $\Omega(T)\subset [\mathcal{B}(T)]^\omega$ is stationary  and
$\mathcal{P}$ is an $\Omega(T)$-complete forcing.
Then  $T$ has no Aronszajn subtree in $\mathbf{V}^\mathcal{P}$.
\end{lem}
 \begin{lem} \label{No New branch}
 Assume $T$ is an $\omega_1$-tree, $X$ is an uncountable set, 
$S \subset [X]^\omega$ is stationary, and 
$\mathcal{P}$ is an $S$-complete forcing. Then $\mathcal{P}$ does not add new cofinal branches to $T$.
 \end{lem}
The following definition is a modification of the Shelah's notion for chain condition, $\kappa$-proper
isomorphism condition. We will be using it for verifying certain chain conditions.
\begin{defn} \label{S-cic}
Assume $S,X$ are as above and  $\kappa$ is an regular cardinal.
 We say that $\mathcal{P}$ satisfies the \emph{$S$-closedness isomorphism condition for 
$\kappa$},
or $\mathcal{P}$ has the \emph{$S$-cic for $\kappa$} ,  if whenever
\begin{itemize}
\item
$M,N$ are suitable models for $\mathcal{P}$,
\item
both $M \cap X, N\cap X$ are in $S$,
\item
$h:M\rightarrow N$ is an isomorphism such that $h\upharpoonright (M\cap N) = id$,
\item
$\textrm{min}(N\setminus M \cap \kappa) > \sup(M\cap \kappa)$, and 
\item
$\langle p_n: n\in \omega \rangle $ is an $(M,\mathcal{P})$-generic sequence,
\end{itemize} 
then there is a common lower bound $q \in \mathcal{P}$ for
$\langle p_n: n\in \omega \rangle $ and $\langle h(p_n): n\in \omega \rangle $.
\end{defn}
\begin{lem} \label{chain}
Assume $2^{\aleph_0} < \kappa$, $\kappa$ is a regular cardinal and that $S,X$ are as above. 
If $\mathcal{P}$ satisfies the $S$-cic for $\kappa$ then it has the
 $\kappa$-c.c.
\end{lem}
The proof of the following fact, which is useful in verifying the chain condition properties of an iteration
of posets, can be found in \cite{second}.
\begin{lem} \label{chain CS}
Suppose 
$\langle \mathcal{P}_i, \dot{\mathcal{Q}}_j: i\leq \delta, j < \delta \rangle$
is a countable support iteration of $S$-complete forcings, where $S\subset [X]^{\omega}$
is stationary and $X$ is uncountable. Assume in addition that
\begin{center}
$\Vdash_{\mathcal{P}_i} "\dot{\mathcal{Q}_i}$ has the $\check{S}$-cic for $\kappa$",
\end{center}
for all $i \in \delta$. Then $\mathcal{P}_\delta$ has the $S$-cic for $\kappa$.
\end{lem}

\section{The generic homogeneous Kurepa tree} \label{tree}
\begin{defn}
Assume $\langle M_\xi :\xi \in \omega_2 \rangle$ is a continuous $\in$-chain of $\aleph_1$-sized 
elementary submodels of $H_{({2^{\omega_2}})^+}$, such that
$\xi \cup \omega_1 \subset M_\xi$ and $\langle M_\eta :\eta  \leq \xi \rangle$ is in $M_{\xi+1}$. 
Fix $C\subset \omega_2$ consisting of all $\sup (M_\xi \cap \omega_2)$ for $\xi \in \omega_2$.
The poset $\mathcal{H}$ is the collection of all conditions $q=(T_q,b_q,\Pi_q)$ 
for which the following statements hold:
\begin{enumerate}
\item
$T_q$ is a countable tree of height $\alpha_q+1$ which is equipped with a lexicographic order such that
for all $t\in (T_q)_{<\alpha_q}$, the set $t^+$, consisting of all immediate successors of $t$,
is isomorphic to the rationals when considered with the lexicographic order.
\item
$b_q$ is a bijective function from a countable subset of $\omega_2$ to the last level of $T_q.$
\item
The collection , 
$\Pi_q = \langle \pi^q_{t,s} :(t,s) \in \bigcup_{\xi \in \alpha_q} ((T_q)_\xi )^2 \rangle$ 
such that $\pi^q_{t,s}$ is a tree isomorphism from $T_q(t)$  to $ T_q(s)$, which preserves the lexicographic order.
\item The collection $\Pi_q$ is coherent, in the sense that 
if $t'>t$ and $\pi^q_{t,s}(t')=s' $ then $\pi^q_{t',s'} = \pi^q_{t,s} \upharpoonright T_q(t')$.
\item The collection $\Pi_q$ is symmetric in the sense that $\pi^q_{s,t}=(\pi^q_{t,s})^{-1}$.
\item The collection $\Pi_q$ respects the club $C$ in the following sense. If $\alpha \in C$,
 $t,s$ are in $T_q$ and have the same height,
then $\xi < \alpha $ iff $b^{-1}_q(\pi^q_{t,s}(b_q(\xi)))< \alpha$. 
\item The collection  $\Pi_q$ respects the composition operation, in the sense that if $t,s,u$ are in $(T_q)_\xi$ and 
$\xi < \alpha_q$ then $\pi^q_{s,u} \circ \pi^q_{t,s}=\pi^q_{t,u}$.
\end{enumerate} 
For $p,q \in \mathcal{H}$ we let $q\leq p$ if 
\begin{enumerate}
\item
$(T_q)_{\leq \alpha_p}=T_p$ and the lex order on $T_p$ is the same as the one on $T_q$,
\item
$\dom(b_p)\subset \dom(b_q)$,
\item
for all $\xi \in \dom(b_p)$, $b_p(\xi)\leq b_q(\xi)$,
\item
for all $(t,s) \in \bigcup_{\xi \in \alpha_q}(T_p)_\xi^2 $, $\pi^p_{t,s}$ is equal to 
$\pi^q_{t,s} \upharpoonright T_p$, and
\item \label{branch index preserving}
for all  $(t,s) \in \bigcup_{\xi \in \alpha_q}(T_p)_\xi^2 $, and $\xi, \eta \in \dom(b_p)$, if
$\pi^p_{t,s}(b_p(\xi))=b_p(\eta)$ then $\pi^q_{t,s}(b_q(\xi))=b_q(\eta)$.
\end{enumerate}
\end{defn}
\begin{notation}
Assume $G$ is a generic filter for $\mathcal{H}$. Define  $T_G$ to be $\bigcup_{q\in G} T_q$ and
$b_\xi$ to be the branch $\{b_q(\xi): q \in G\}$. If $t,s$ are in $T_G$ and have the same height 
$\pi_{t,s}$ denotes $\bigcup_{q\in G} \pi^q_{t,s}$.
\end{notation}

\begin{lem} \label{ctbly closed}
$\mathcal{H}$ is $\sigma$-closed. 
\end{lem}
\begin{proof}
Let $\langle p_n:n\in \omega  \rangle$ be a decreasing sequence in $\mathcal{H}$ and
$\sup(\alpha_{p_n})_{n \in \omega} = \alpha$. Let $T=\bigcup_{n\in \omega} T_{p_n}$. 
Note that $(b_{p_n}(\xi): n \in \omega)$ is a cofinal chain in $T$ for all
$\xi \in \bigcup_{n \in \omega} \dom(b_{p_n})$.
Let $T_q$ be a countable tree of height $\alpha +1$ such that
\begin{itemize}
\item $(T_q)_{<\alpha}=T$,
\item for all $\xi \in \bigcup_{n \in \omega} \dom(b_{p_n}) $, $(b_{p_n}(\xi): n \in \omega)$ has an
upper bound in $T_q$, and
\item every element of height $\alpha$ is an upper bound for $(b_{p_n}(\xi): n \in \omega)$,
for some $\xi \in \bigcup_{n \in \omega} \dom(b_{p_n}) $.
\end{itemize}
Now let $q$ be the condition with $\alpha_q=\alpha$ and  $T_q$ as above. 
Let $b_q$ be the function form $\bigcup_{n\in \omega} \dom(b_{p_n})$ to the last level of $T_q$
such that for all $\xi$ in the domain, $b_q(\xi)$ is the upper bound for the chain 
$(b_{p_n}(\xi): n \in \omega)$.
Similarly $\bigcup_{n\in \omega} \pi^{p_n}_{t,s}$, can be extended to the last level of 
$T_q$, for all $t,s$ that are of the same height and are in $T$. 
It is easy to see that the condition $q$ described above is a lower bound for the sequence 
$\langle p_n:n\in \omega  \rangle$. 
\end{proof}
\begin{lem}
$\GCH$ implies that $\mathcal{H}$ has the $\aleph_2$-cc.
\end{lem}
\begin{proof}
Let $\langle q_\xi: \xi \in \omega_2 \rangle $ be a collection of conditions in  $\mathcal{H}$.
Since there are $\aleph_1$-many possibilities for $T_q$ and $\Pi_q$,
we can thin down this collection to a subset of the 
same cardinality so that $T_{q_\xi}$ and $\Pi_{q_\xi}$ do not depend on $\xi$.
Now define $f:C \longrightarrow \omega_2$ by $f(\xi)=\sup(\dom(b_{q_\xi})\cap \xi)$, where $C$ is the club that 
all elements of $\mathcal{H}$ respect. Note that for all $\xi \in C$ with $cf(\xi)>\omega$, $f(\xi) < \xi$.
So there is a stationary $S\subset C$, and $\alpha\in \omega_2$ such that $f\upharpoonright S$ is the 
constant $\alpha$. We can thin down $S$ to a stationary subset $S'$ if necessary, so that 
in $\langle q_\xi: \xi \in S' \rangle $,
$\dom(b_{q_\xi})\cap \alpha$ and $b_{q_\xi}\upharpoonright \alpha$ do not depend on $\xi$.
Let $S''\subset S' \setminus (\alpha+1)$ 
be of size $\aleph_2$ and whenever $\xi < \eta $ are in $S''$, 
$\sup(\dom(b_{q_\xi}))<\eta$. 
Note that $\langle b_{q_\xi}: \xi \in S'' \rangle$ forms a $\Delta$-system with root $r$
such that the $\dom(r)\subset \alpha$. Moreover for all $\xi \in S''$, 
$\min(\dom(b_{q_\xi})\setminus r)\geq \xi$.
Since $S''\subset C$, every two conditions in $\langle q_\xi: \xi \in S'' \rangle $ are compatible.
\end{proof}

The following can routinely be verified.

\begin{fact}
The following sets are dense in $\mathcal{H}$.
\begin{itemize}
\item $H_\alpha:=\{q \in \mathcal{H}| \alpha_q>\alpha\}$.
\item For $\xi \in \omega_2,$  $I_\xi:=\{ q \in \mathcal{H}: \xi \in \dom(b_q) \}$.
\end{itemize}
\end{fact}

The proof of the following lemma is the same as Lemma \ref{ctbly closed}.

\begin{lem}
If $M$ is suitable for $\mathcal{H}$ and $\langle p_n: n \in \omega \rangle$ is a decreasing 
$(M,\mathcal{H})$-generic sequence, then there is a lower bound $q$ for 
$\langle p_n:n\in \omega \rangle$ such that 
$\dom(b_q)=M\cap \omega_2$, and $\alpha_q=M\cap \omega_1$. 
\end{lem}

\begin{fact} \label{nice tree}
\begin{itemize}
\item Assume $G$ is a generic filter for $\mathcal{H}$. Then the generic tree
$T:= \bigcup_{q\in G}T_q$ is a Kurepa tree such that  
$\langle \{b_q(\xi): q \in G \} : \xi \in \omega_2 \rangle$
is an enumeration of the set of all branches.
\item $T$ has no Aronszajn subtree. Moreover, any uncountable downward closed subtree of $T$
contains a branch $b_\xi$ for some $\xi \in \omega_2$.
\item  Assume $L$ is the linear order consisting of all branches of $T$, $\mathcal{B}(T)$,
ordered by the lexicographic order of the tree $T$. Then $\Omega(L)$ is stationary.
\end{itemize}
\end{fact}
\begin{proof}
The first statement follows from the second one.
The second statement follows from the third one and Proposition 2.9, or the stronger statement Theorem 4.1 of 
\cite{no_real_Aronszajn}.
For the last statement, let $M$ be suitable
for $\mathcal{H}$ and $p\in M\cap \mathcal{H}$. Then the $(M,\mathcal{H})$-generic condition from the last
lemma forces that $M[\dot{G}]\cap  L \in \Omega(L)$. 
\end{proof}

From now on, $T$ is the generic Kurepa tree generated by $\mathcal{H}$ unless otherwise mentioned.
Also $K$ is the linear order $\mathcal{B}(T)$ ordered by the lexicographic order 
of the tree $T$. We fix an enumeration of 
$\mathcal{B}(T)=\langle b_\xi : \xi \in \omega_2 \rangle$.

The rest of this section is devoted to showing that $K$ has a lot of non-$\sigma$-scattered suborders
that are amenable. These facts are not used in the proof of the results in the next sections but show 
some possible obstruction for the minimality of suborders of $K$. In the next sections, these 
non-$\sigma$-scattered suborders are forced to be $\sigma$-scattered by an improper forcing.
Here we say a countable sequence of conditions in $\mathcal{H}$ forces a statement
 if every lower bound of that sequence 
forces that statement, equivalently all generic filter that contain the sequence extends the model to the 
one in which the statement holds.

\begin{defn}
Let $T$ be the $\mathcal{H}$-generic tree and $t\in T$. The element $t$ is said to be \emph{simple} if whenever 
$\theta> 2^{\omega_1}$ is a regular cardinal and 
$M$ is a countable elementary submodel of $H_\theta$ containing $T$, then
$M$ captures $t \in T$. Otherwise $t$ is said to be \emph{complex}.
\end{defn} 

\begin{lem}
Assume $\GCH$ holds in $\mathbf{V}$, $M$ is suitable for $\mathcal{H}$, $\langle p_n: n\in \omega \rangle $
is an $(M, \mathcal{H})$-generic sequence, $t \in T_0:=\bigcup_{n \in \omega} T_{p_n}$,
$\langle p_n \rangle_{n \in \omega} \Vdash "t$ is simple", $b$ is a branch in $T_0$
and $\textrm{ht}(t)<\alpha< \delta:=M\cap \omega_1$. Then there exists $s \in T_0$ such that $\textrm{ht}(s)=\alpha$,
$t<s$, $s \notin b$ and $\langle p_n \rangle_{n \in \omega} \Vdash "s$ is simple".
\end{lem}
\begin{proof}
First note that if $G$ is $\mathcal{H}$-generic over $\mathbf{V}$ then
$H_{\omega_3}[G]=H^{\mathbf{V}[G]}_{\omega_3}$
has a well ordering $\lhd$. Let $\dot{\lhd}$ be an $\mathcal{H}$-name for $\lhd$.
Since $\langle p_n \rangle_{n \in \omega}$ is $M$-generic it decides 
$\dot{\lhd} \cap (M[\dot{G}])^2$, in the sense that, if $\tau$ and $\pi$ are two 
$\mathcal{H}$-names that are in $M$ then there is an $n \in \omega$ such that
$p_n \Vdash ``\tau \dot{\lhd} \pi"$ or $p_n \Vdash "\pi \dot{\lhd} \tau"$.

Also note that if $t$ is simple then so is every $t'\in t^+.$ Now let $\sigma \in M$
be an $\mathcal{H}$-name for a branch of the $\mathcal{H}$-generic tree such that
$\langle p_n \rangle_{n \in \omega}$ forces that
\begin{itemize}
\item  $t \in \sigma$,
\item $\sigma(\textrm{ht}(t)+1)\neq b(\textrm{ht}(t)+1)$, and 
\item $\sigma$ is the $\dot\lhd$-minimum branch of $\dot{T}$ with the properties above.
\end{itemize}
Let $s \in T_0$ such that $\langle p_n \rangle_{n \in \omega}$ forces that $s=\sigma(\alpha)$. We will show 
that $\langle p_n \rangle_{n \in \omega} \Vdash " s$ is simple". Let $G$ be an $\mathcal{H}$-generic filter
containing $\langle p_n \rangle_{n \in \omega}$ and in $\mathbf{V}[G]$, $N$ be a countable elementary 
submodel of $H_{\omega_3}$. If $N\cap \omega_1 \leq \textrm{ht}(t)$, by simplicity of $t$, $N$ captures $s.$
if $\textrm{ht}(t)< N\cap \omega_1$ then $t^+ \subset N$ so $\sigma_G = 
\textrm{min}_{\lhd}\{ c \in \mathcal{B}(T): c(\textrm{ht}(t)+1)=s(\textrm{ht}(t)+1) \}.$ So by elementarity 
$\sigma_G \in N$ and $N$ captures $s.$ 
\end{proof}

\begin{prop} \label{small big suborders}
Assume $\GCH$ holds in $\mathbf{V}$ and $G$ is $\mathbf{V}$-generic for $\mathcal{H}$.
Then $K$ has an amenable non-$\sigma$-scattered suborder in $\mathbf{V}[G]$.
\end{prop}
\begin{proof}
Let $L=\{t \in T : t$ is minimal complex$\}$ ordered by the lexicographic order of the $\mathcal{H}$-generic
tree $T$. To see $L$ is amenable, let $M$ be  a  countable elementary submodel of $H_\theta$ 
with $T,L \in M$, where $\theta$ is a regular large enough cardinal and $t\in L$. 
We need to show that $t$ is internal to $M$. 
If $\textrm{ht}(t) < M\cap \omega_1$ then $t \in M$ and there is nothing to prove.
If $\textrm{ht}(t)> M\cap \omega_1$, note that $t \upharpoonright (M \cap \omega_1)$ is simple and $M$ captures it. 
If $\textrm{ht}(t) = M\cap \omega_1$,
let $E=\{N\cap \mathcal{B}(T):N$ is a countable elementary submodel of $H_{\omega_3}$ with $T,L \in N \}$.
For every $Z \in E\cap M$ there is a countable elementary submodel $N$ of 
$H_{\omega_3}$ such that $N\in M$ and $N\cap \mathcal{B}(T)=Z$. 
In particular $N \cap \omega_1 < \textrm{ht}(t)$, and 
since $ t$ is a minimal complex element of $T$,
$Z$ captures $t$. So $t$ is internal to $M$ and $L$ is amenable.

In order to see $L$ is not $\sigma$-scattered we will show that $\Gamma(L)$ is stationary 
in $[\hat{L}]^\omega$. Assume $\dot{E}$ is an $\mathcal{H}$-name  for a club in $[\hat{L}]^\omega$ 
and $q \in \mathcal{H}$.
In $\mathbf{V}$, 
let $M$ be suitable for $\mathcal{H}$ with $q , \dot{E}$ in $M$ and 
$\langle p_n: n \in \omega \rangle$ be an $M$-generic sequence such that $p_0=q.$
Also let $\langle b_n : n \in \omega \rangle$ be an enumeration of all branches of 
$T_0=\bigcup_{n \in \omega}T_{p_n}$ which are downward closure of $\{b_{p_n}(\xi): n \in \omega \}$ for some $\xi \in M\cap \omega_2$.
By the previous lemma there is a sequence 
$\langle t_k: k \in \omega \rangle$ of elements in $T_0$ such that 
for all $k \in \omega$, $\langle p_n \rangle_{n \in \omega}$ forces that $t_k$ is simple, $t_k < t_{k+1}$,
$t_k \notin b_k$ and $\sup\{\textrm{ht}(t_k): k \in \omega \}= \delta:=M\cap \omega_1$.

Now let $T_p=T_0 \cup (T_p)_\delta$ where $ (T_p)_\delta$ is a minimal set such that
\begin{itemize}
\item 
for each $\xi \in M\cap \omega_2$, $\{b_{p_n}(\xi) : n \in \omega \}$ has a unique upper bound in $(T_p)_\delta$,
\item
the sequence $\langle t_k: k \in \omega \rangle$ has a unique upper bound for in $(T_p)_\delta$, and 
\item
for each $u,v \in T_0$ and $t \in (T_p)_\delta$ with $u < t$, $(\pi^{p_n}_{u,v})[\{s\in T_0:s<t\}]$ has a unique upper bound 
in $(T_p)_\delta$.
\end{itemize}
It is easy to see that there is a 
$b_p$ which is a  function  from a countable subset of $\omega_2$ to $(T_p)_\delta$ and 
$\Pi_p$ consisting of natural extensions of the maps $\pi^{p_n}_{u,v}$ where $u,v$ are in $T_0$, such that
$p=(T_p,b_p,\Pi_p)$ is a lower bound for $p$.

On the other hand $p$ forces the following statement.
\begin{itemize}
\item
There are minimal complex elements at the $\delta$th level of the $\mathcal{H}$-generic tree $T$.
\item
$M[\dot{G}]\cap \tau \in \dot{E}$, where $\tau$ is an $\mathcal{H}$-name for $\hat{L}$ in $M$.
\item
$M[\dot{G}]\cap \tau$ does not capture all elements of $\dot{L}$.
\end{itemize}
Therefore $\mathbb{1}_\mathcal{H}\Vdash ``\dot{L} $ is not $\sigma$-scattered." Note that the elements of $L$
form an antichain in $T$. Let $L'\subset K$ such that for every $t \in L$ there is a unique branch  
$b \in L'$ with $t \in b$. Then $L'$ is isomorphic to $L$, hence $K$ has an amenable non-$\sigma$-scattered
suborder.
\end{proof}

\section{Adding embeddings} \label{embeddings}

In the previous section we introduced a forcing which generates a Kurepa tree $T$ equipped with 
a lexicographic order which also has some homogeneity properties.
In this section we use the homogeneity of $T$ to prove the countable support iteration of
some forcings that add embeddings among the $\aleph_1$-sized
dense subsets of the linear order $K=(\mathcal{B}(T),\lex)$ do not collapse cardinals.
We fix an enumeration $\langle b_\xi : \xi \in \omega_2 \rangle$
of the branches of the tree $T$ for the rest of the paper, and recall that for each $t\in T$, the set $t^+$, consisting of all 
immediate successors of $t$ with respect to $<_T$, is isomorphic to the rationals when considered with $\lex$. 
Here homogeneity of $T$
means that there is a collection $\Pi = \langle \pi_{t,s}: t,s \in T \wedge \textrm{ht}(t)=\textrm{ht}(s) \rangle$ with the following 
properties.

\begin{enumerate}
\item for all $t,s$ in $T$ which have the same height,
$\pi_{t,s}$ is a tree and lex order isomorphism from  $T(t)$ to $T(s)$.
\item $\Pi$ is symmetric, in the sense that $\pi_{t,s} = (\pi_{s,t})^{-1}$.
\item $\Pi$ is coherent in the sense that if $t,s,t',s'$ are in $T$, $\textrm{ht}(t)=\textrm{ht}(s)$, $t<t'$,
$s<s'$ and $\pi_{t,s}(t')=s'$, then $\pi_{t,s}\upharpoonright(T(t')) = \pi_{t',s'} $, where $t'\uparrow = \{u \in T : t'$ is 
compatible with $ u \}$. 
\end{enumerate}

\begin{defn} \label{embedding forcing}
Assume $T$ is as above and $X,Y$ are two   subsets of $\omega_2$
such that $|X|=|Y|=\aleph_1$ and both $\{ b_\xi : \xi \in X \}$ , $\{ b_\xi : \xi \in Y \}$ are dense in $K$. 
$\mathcal{F}_{XY}(=\mathcal{F})$ is
the poset consisting of all conditions $p=(f_p,\phi_p)$ for which the following holds:
\begin{enumerate}
\item $f_p: T\upharpoonright A_p \longrightarrow T\upharpoonright A_p$ is a lex order and level preserving 
tree isomorphism 
where $A_p \subset \omega_1$ is countable and closed with $\textrm{max}A_p=\alpha_p$.
\item $\phi_p$ is a countable partial injection from $\omega_2$ to $\omega_2$ such that:
\begin{enumerate}
\item for all $\xi \in \dom(\phi_p)$,  if $\xi \in X$ then $ \phi_p(\xi) \in Y$,
\item \label{2b}
 for all $\xi \in \dom(\phi_p)\setminus X$, $b_{\phi_p(\xi)}=\pi_{t,s} [b_\xi]$,
where $t = b_\xi (\alpha_p +1)$  and $s$ is an immediate successor of $f_p(b_\xi(\alpha_p))$, and 
\item the map $b_\xi \mapsto b_{\phi_p(\xi)}$ is lexicographic order preserving.
\end{enumerate}
\item For all $t \in T_{\alpha_p}$ there are at most finitely many $\xi \in \dom(\phi_p)$ with $t \in b_\xi$.
\item For all $\xi \in \dom(\phi_p)$, $f_p(b_\xi(\alpha_p))=b_{\phi_p(\xi)}(\alpha_p)$.
\end{enumerate}
We let $q\leq p$ if $f_p \subset f_q$, $A_q \cap \alpha_p = A_p$ and $\phi_p \subset \phi_q$.
\end{defn}

It is obvious that the sets 
$\{q \in \mathcal{F}: \alpha_q > \beta \}$ 
and $\{q \in \mathcal{F}: \xi \in \dom(\phi_q) \}$ are dense for all 
$\beta \in \omega_1$ and $\xi \in \omega_2$. Therefore the forcing $\mathcal{F}$ 
adds a  lexicographic order 
embedding from $X$ to $Y$ via the map $\Phi \upharpoonright X$ where 
$\Phi = \bigcup_{p \in G} \phi_p$ and $G$ is 
the generic filter for $\mathcal{F}$. 
We will show that countable support iterations of these forcings do not collapse cardinals. 
\begin{lem} \label{robust forcing}
Assume $\mathcal{P}$ is an $S$-complete forcing where $S=\Omega(T)$, and $\dot{X}, \dot{Y}$ are $
\mathcal{P}$-names 
for the indexes of the elements of $\aleph_1$-sized dense subsets of $K$. Then
\begin{enumerate}
\item$\Vdash ``\dot{\mathcal{F}_{XY}}$ is $\check{S}$-complete", and 
\item$\Vdash ``\dot{\mathcal{F}_{XY}}$ has the $\check{S}$-cic for $\check{\omega}_2$"
\end{enumerate}
\end{lem}
\begin{proof}
Let $G \subset \mathcal{P}$ be generic. We work in $\mathbf{V}[G]$. 
To see (1), assume $M$ is suitable for 
$\mathcal{F}$ and $M \cap K \in S$. 
Also let $\langle p_n=(f_n,\phi_n) : n \in \omega \rangle$ be a descending 
$(M, \mathcal{F})$-generic sequence and $\delta = M\cap \omega_1$. Note that 
$M\cap \omega_2=\bigcup_{n \in \omega} \dom(\phi_n)$ 
and $\bigcup_{n \in \omega} A_{p_n}$ is cofinal in $\delta$.
Now let $\phi_p=\bigcup_{n \in \omega} \phi_n$, and $f_p=\bigcup_{n \in \omega} f_n \cup f$, where 
$f(b_\xi(\delta))=b_{\phi_p(\xi)}(\delta)$ for all $\xi \in M \cap \omega_2$. 
This makes $p$ a lower bound for $\langle p_n: n \in \omega \rangle$,
since $M\cap K \in S$, and $\{b_\xi(\delta): \xi \in M\cap \omega_2 \}=T_\delta$.  

For (2), still in $\mathbf{V}[G]$, let $M,N, \langle p_n=(f_n,\phi_n) : n \in \omega \rangle$, 
and $h$ be as in 
Definition \ref{S-cic} with $M\cap \omega_1 = N\cap \omega_1 = \delta$. 
Since $h$ fixes the intersection $h(f_n)=f_n$ and
$b(\delta)=[h(b)](\delta)$, for all $b \in M\cap \mathcal{B}(T)$. 
Let $\phi_p=\bigcup_{n \in \omega} (\phi_n \cup h(\phi_n))$ 
and $f_p=\bigcup_{n \in \omega} f_n \cup f$, where 
$f(b_\xi(\delta))=b_{\phi_p(\xi)}(\delta)$ for all $\xi \in M \cap \omega_2$.

We need  to show that 
$p$ is a condition and a common lower bound for $\langle p_n:n\in \omega \rangle$ 
and its image under $h$. 
We will  prove  the map $b_\xi \mapsto b_{\phi_p(\xi)}$ preserves the order $\lex$. 
The rest of the requirements are obvious.
Let $\xi, \eta$ be in $(M \cup N) \cap \omega_2$ and $b_\xi \lex b_\eta$.
If one of $\xi$ or $\eta$ is in $M \cap N$, we are done. 
We are also done if $b_\xi(\delta) \neq b_\eta(\delta)$.
So assume that $\xi \in M $, $\eta \in N$, and  $b_\xi(\delta) = b_\eta(\delta)$.
By elementarity $\eta = h(\xi)$.
Fix $n \in \omega$ such that $\xi \in \dom(\phi_n)$.
Since $|X|= \aleph_1$ and $X \in M \cap N$, $M\cap X = N \cap X$. In particular, $\xi , \eta $ are not in $X$.
Let $t = b_\xi (\alpha_{p_n} +1)$  and $s$ be the  immediate successor of $f_{p_n}(b_\xi(\alpha_{p_n}))$  such that
$b_{\phi_n(\xi)} = \pi_{t,s}[b_\xi]$. Then $b_{h \phi_n (\eta)} = \pi_{t,s}[b_\eta]$. But $\pi_{t,s}$ preserves 
$\lex$, so $b_{\phi_n(\xi)} \lex b_{h \phi_n (\eta)} $. Hence, $\phi_p$ preserves $\lex$.
\end{proof}

\section{Proof of the main theorem} \label{last}

In this section we will finish the proof of Theorem \ref{main}. The strategy is to show that if
two $\aleph_1$-sized $L,L' \subset K$ have closure of cardinality $\aleph_2$, then they are isomorphic. 
Note that by Lemma \ref{small big suborders}, 
$K$ has non-$\sigma$-scattered suborders whose closure have 
cardinality $\aleph_1$. So in order to use the strategy mentioned above, we need to make these suborders 
$\sigma$-scattered by forcings for which the analogue of Lemma \ref{robust forcing} holds. We finish this 
section with a proof of Theorem \ref{main}.

\begin{defn}
Assume $L\subset K$, $|\bar{L}| \leq \aleph_1$. $\mathcal{P}_L(=\mathcal{P})$ is the poset consisting of 
all conditions $p:\alpha_p+1 \longrightarrow [\bar{L}]^\omega \cap \Omega(L)$ that are $\subset$-
increasing and continuous. 
\end{defn}

\begin{lem} \label{shooting}
Assume $S=\Omega(K)$, $\mathcal{Q}$ is an $S$-complete forcing, 
and $\dot{L}$ is a $\mathcal{Q}$-name 
for a suborder of $K$ whose closure has size $\leq \aleph_1$. Then
\begin{enumerate}

\item  $\Vdash ``\dot{\mathcal{P}_{L}}$ is $\check{S}$-complete",and 
\item $\Vdash ``\dot{\mathcal{P}_{L}}$ has the $\check{S}$-cic for $\check{\omega}_2$".

\end{enumerate}
\end{lem}
\begin{proof}
Let $G \subset \mathcal{Q}$ be generic. We work in $\mathbf{V}[G]$. To see (1), let $M$ be suitable for 
$\mathcal{P}$ and $M \cap K \in S$. It is enough to show that $M\cap \bar{L} \in \Omega(L)$. 
First note that $M$ does not capture any $x \in K \setminus M$ via
 cuts of countable cofinality or initiality. 
In order to see this, assume $M$ captures $x \in K$ via a cut $z$ where $z = \sup\{y_n : n \in \omega\}$
and $\langle y_n : n \in \omega \rangle$ is an increasing sequence in $K.$
Let $\alpha = \sup \{x \Delta y_n : n \in \omega \}$. Obviously $\alpha \in M \cap \omega_1$ and
we can find $x' \in K \cap M$ that is strictly in between $z $ and $x$. This contradicts the assumption that
$M$  captures $x$ via $z$.
 So if $M$ captures an element that is not in $M$, it has to capture it via a cut $z\in \hat{K}$
 of cofinality and coinitiality 
$\aleph_1$. But then $z$ determines a branch in $T$ which means that $z \in K$.

Now let $M$ capture $x \in L \setminus M$ via $z\in K\cap M$. 
We will show that $z \in \bar{L}$.
Note that $K \setminus \bar{L}$ is 
the union of a collection consisting of pairwise disjoint convex open subsets of $K$. So if 
$z \in (K \setminus \bar{L}) \cap M $ there is a convex open set $I$ containing $z$ which is in $ M$. 
Since $I \in M$ the endpoints of $I$ are in $M \cap \hat{K}$. But $M$ captures $x$ via a unique 
cut, so $z$ is an endpoint of $I$ which contradicts the fact that $I$ is open. 

For (2), note that if $h:M \longrightarrow N$ is an isomorphism that fixes $M\cap N$, then 
$h$ fixes $\bar{L}\cap M$ because $|\bar{L}|=\aleph_1$. So any lower bound for an $M$-generic sequence 
is a lower bound for an $N$-generic sequence. 
\end{proof}

Now we are ready to prove Theorem \ref{main}. Assume $\GCH$ holds in $\mathbf{V}$ and $T$ is the generic 
Kurepa tree from the forcing $\mathcal{H}$ in $\mathbf{V}^\mathcal{H}$. By Facts and Lemmas 
\ref{No A subtree}, \ref{No New branch}, \ref{nice tree}, \ref{robust forcing}, \ref{shooting}, and the work in \cite{second}
 there is countable support
iteration of forcings of length $\omega_2$ which is $\Omega(T)$-complete and extends $\mathbf{V}^\mathcal{H}$
to a model in which the following holds.
\begin{enumerate}
\item $T$ is club isomorphic to all of its everywhere Kurepa subtrees and has no Aronszajn subtree.
\item If $X,Y$ are two dense suborders of $K=(\mathcal{B}(T),\lex)$ and $|X|=|Y|=\aleph_1$ then 
$X$ embeds into $Y$ as a linear order.
\item  If $X\subset K$ and $|\bar{X}|\leq \aleph_1$ then $X$ is $\sigma$-scattered.

\end{enumerate}
Note that if $L\subset K, |L|=\aleph_1, |\bar{L}|=\aleph_2$, then there is $L_0\subset L$ such that 
$\bar{L}_0$ is $\aleph_2$-dense. 
To see this, for $b,b' \in L$, let 
$b \sim b'$ if there are at most 
$\aleph_1$ many elements of $\bar{L}$ in between $b,b'$. It is obvious that there are at least two 
distinct equivalence classes. 
We consider the set of equivalence classes as a linear order.
Here, the equivalence classes are ordered by the order of their elements. 
Since the equivalence classes are convex subsets of $L$, this order is well defied. 

The set of equivalence classes is $\aleph_1$-dense.
In order to see this,
let $b,b'$ be two non-equivalent elements of $L$ such that there are only 
countably many equivalence classes in between them.
Note that these equivalence classes are disjoint convex sets.
Let $\alpha \in \omega_1$ be large enough such that for each $t \in T_\alpha$ with 
$b \lex t \lex b'$, the set of all branches containing $t$ intersects at most one equivalence class. 
For each $c \in \bar{L} \cap (b,b')$ there exists $t \in T_\alpha$ with $b \lex t \lex b'$ such that $t \in c$.
So, there are only $\aleph_1$ many elements of $\bar{L}$ in between $b,b'$, which is a contradiction.
Now, let $L_0$ be a suborder of $L$ that intersects each equivalence class at  exactly one point. 
$\bar{L}_0$ is $\aleph_2$-dense. 
In order to see this, let $b \in \bar{L} \setminus \bar{L}_0.$
Fix $t \in b$ such that $\mathcal{B}(T_t)$ intersect at most one equivalence class.
Since, $| \bar{L} \cap \mathcal{B}(T_t) | \leq \aleph_1$ and there are at most $\aleph_1$ many such $t \in T$,
$|\bar{L} \setminus \bar{L}_0| \leq \aleph_1$.

Note that for such an $L_0$, the tree $\bigcup \bar{L}_0$ is an everywhere Kurepa subtree of $T$. So $L_0$
is isomorphic to an $\aleph_1$-sized dense suborder of $K$. This finishes the proof because all 
$\aleph_1$-sized dense suborders of $K$ are biembeddable.

We will finish the paper with some remarks about the iteration of the forcings we used.
The most important features of the forcings we used are $\Omega(T)$-completeness
and $\aleph_2$-chain conditions.
These forcings preserve the stationarity of stationary subsets of $\Omega(T)$, but they do not need to 
preserve the stationarity of stationary subsets of $\Gamma(T)$. In fact, some of the iterands we considered
shoot clubs into the complement of some stationary subsets of $\Gamma(T)$.
On the other hand the set $\Gamma(T)$
itself remains stationary in the final model we obtain, by Proposition \ref{char}. The only way to see that $\Gamma(T)$ is stationary is that
$\omega_2$ is preserved and consequently $K$ is not $\sigma$-scattered.
The phenomenon that only preserving $\omega_2$ without any control on countable 
structures which come from $\Gamma(T)$ guarantees that $\Gamma(T)$ remains stationary seems to be new 
and mysterious.
For instance, assume $S\subset \Gamma(T)$ is stationary and is not in the form of $\Omega$ or $\Gamma$ of any
suborder of $K$. Is there any way to determine whether or not $S$ remains stationary in the extension
under counatble support iterations of these forcings?  
\section*{Acknowledgments}
For continual guidance and encouragement and for
suggesting this project the author would like to thank Justin Tatch Moore.

The research presented in this paper was supported in part by NSF grants DMS-1262019 and DMS-1600635.

%\bibliography{global}{}
%\bibliographystyle{abbrv}

\def\Dbar{\leavevmode\lower.6ex\hbox to 0pt{\hskip-.23ex \accent"16\hss}D}

\end{document}